\documentclass[a4paper,12pt]{article}
\usepackage[utf8]{inputenc}
\usepackage[T1]{fontenc}
\usepackage{amsfonts,amsmath,amssymb,amsthm}
\theoremstyle{definition}

\theoremstyle{plain}
\newtheorem{corollary}{Corollary}
\newtheorem{lemma}{Lemma}
\newtheorem{proposition}{Proposition}
\newtheorem{theorem}{Theorem}
\usepackage{graphicx}
\usepackage{caption}
\usepackage{pgf,tikz,circuitikz,subfig}
\usetikzlibrary{arrows,shapes,positioning}
\usepackage{booktabs,multirow,tabularx}
\usepackage{hyperref}
\hypersetup{colorlinks,%
	citecolor={red}, 
	urlcolor={blue},
	linkcolor={blue},
	breaklinks={true},
	pagebackref={true},
	hyperindex={true},
}
\usepackage{url}
\usepackage{cite}
\usepackage[left=2cm,right=2cm,top=2cm,bottom=2cm]{geometry}
\usepackage[ruled]{algorithm}
\usepackage{algorithmic}
\usepackage{newtxtext}
\usepackage{courier}
\usepackage{enumitem}
\newlist{abbrv}{itemize}{1}
\setlist[abbrv,1]{label=,labelwidth=0.9in,align=parleft,noitemsep,leftmargin=!}
\newcommand{\R}{\mathbb{R}}

\newcommand{\rv}[1]{\boldsymbol{#1}}

\newcommand{\ub}[1]{\overline{#1}}
\newcommand{\lb}[1]{\underline{#1}}
\newcommand{\geo}[1]{\mathtt{#1}}
\DeclareMathOperator{\subj}{s.t.}

\title{The equilateral small octagon of maximal width}
\author{Christian Bingane\thanks{D\'{e}partement de math\'{e}matiques et de g\'{e}nie industriel, Polytechnique Montr\'{e}al, Montreal, Quebec, Canada, H3C~3A7. Emails: \url{christian.bingane@polymtl.ca}, \url{charles.audet@polymtl.ca}} \and Charles Audet\footnotemark[1]}
\begin{document}
\maketitle
\begin{abstract}
A small polygon is a polygon of unit diameter. The maximal width of an equilateral small polygon with $n=2^s$ vertices is not known when $s \ge 3$. This paper solves the first open case and finds the optimal equilateral small octagon. Its width is approximately $3.24\%$ larger than the width of the regular octagon: $\cos(\pi/8)$. In addition, the paper proposes a family of equilateral small $n$-gons, for $n=2^s$ with $s\ge 4$, whose widths are within $O(1/n^4)$ of the maximal width.
\end{abstract}


\section{Introduction}
The {\em diameter} of a polygon is the largest Euclidean distance between pairs of its vertices. A polygon is said to be {\em small} if its diameter equals one. The {\em width}­ of a polygon for a given direction is the distance between two parallel lines perpendicular to this direction and supporting the polygon from below and above. The width of a polygon is the minimum width over all directions. The diameter graph of a small polygon is defined as the graph with the vertices of the polygon, and an edge between two vertices exists only if the distance between these vertices equals one. Figures~\ref{figure:4gon}, \ref{figure:6gon} and \ref{figure:8gon} represent diameter graphs of small quadrilaterals, hexagons and octagons, respectively. The solid lines are the edges of the graphs, and the dashed lines simply delimit the polygons.

For an integer $n \ge 3$, the maximal width of an equilateral small $n$-gon is unknown when $n = 2^s$ and $s \ge 3$~\cite{bezdek2000}. The present paper solves the first open case and gives the optimal equilateral small $8$-gon. Thus, the main result is the following:

\begin{theorem}\label{thm:F8}
	If $w_8^*$ denote the maximal width among all equilateral small $8$-gons then
	\[
	w_8^* = \frac{t_0^2+t_0}{2 \sqrt{t_0^3+t_0^2-2t_0 +1}} + \varepsilon =  0.9537763006\ldots + \varepsilon,
	\]
	where $t_0 = 0.7682191676\ldots$ is the unique positive root of the polynomial equation $t^5-6t^3+3t^2+10t-7=0$ and $0 \le \varepsilon \le 10^{-6}$. The value $w_8^*$ is only attained by the $8$-gon $\geo{F}_8$ illustrated in Figure~\ref{figure:8gon:F8}.
\end{theorem}

In addition, for $n=2^s$ with integer $s\ge 4$, a tight lower bound on the maximal width is proposed by constructing a family of equilateral small $n$-gons whose widths are greater than that of the regular small $n$-gons.

\begin{theorem}\label{thm:Gn}
	Suppose $n=2^s$ with integer $s\ge 4$. Let $\ub{W}_n := \cos \frac{\pi}{2n}$ denote an upper bound on the width $W(\geo{P}_n)$ of a small $n$-gon $\geo{P}_n$~\cite{bezdek2000}. Let $\geo{R}_n$ denote the regular small $n$-gon. Then there exists an equilateral small $n$-gon $\geo{G}_n$ such that
	\[
	\ub{W}_n - W(\geo{G}_n) = \frac{2\pi^4}{3n^4} + O\left(\frac{1}{n^6}\right)
	\]
	and
	\[
	W(\geo{G}_n) - W(\geo{R}_n) = \frac{3\pi^2}{8n^2} + O\left(\frac{1}{n^4}\right).
	\]
\end{theorem}

The remainder of this paper is organized as follows. Section~\ref{sec:ngon} recalls principal results on the maximal width of small polygons. Section~\ref{sec:F8} studies the equilateral small octagon and gives the proof of Theorem~\ref{thm:F8}. Theorem~\ref{thm:Gn} is proved in Section~\ref{sec:Gn} and Section~\ref{sec:conclusion} concludes the paper.

\begin{figure}[H]
	\centering
	\subfloat[$(\geo{R}_4,0.707107)$]{
		\begin{tikzpicture}[scale=4]
			\draw[dashed] (0,0) -- (0.5000,0.5000) -- (0,1) -- (-0.5000,0.5000) -- cycle;
			\draw (0,0) -- (0,1);
			\draw (0.5000,0.5000) -- (-0.5000,0.5000);
		\end{tikzpicture}
	}
	\subfloat[$(\geo{Q}_4,0.866025)$]{
		\begin{tikzpicture}[scale=4]
			\draw[dashed] (0.5000,0.8660) -- (0,1) -- (-0.5000,0.8660);
			\draw (0,1) -- (0,0) -- (0.5000,0.8660) -- (-0.5000,0.8660) -- (0,0);
		\end{tikzpicture}
		\label{figure:4gon:Q4}
	}
	\caption{Two small $4$-gons $(\geo{P}_4,W(\geo{P}_4))$: (a) Regular $4$-gon; (b) An optimal non-equilateral $4$-gon~\cite{tamvakis1987}}
	\label{figure:4gon}
\end{figure}
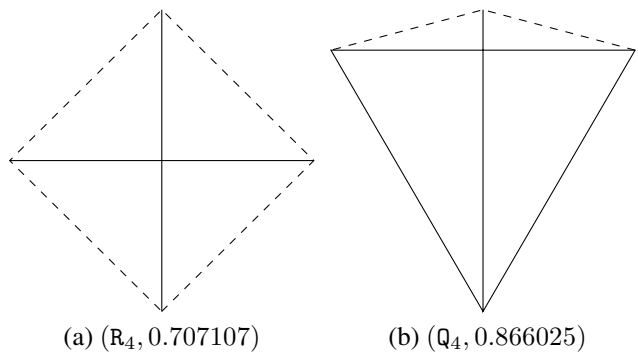

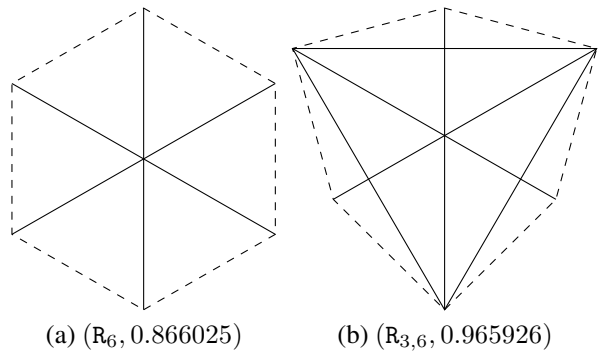
\begin{figure}[H]
	\centering
	\subfloat[$(\geo{R}_6,0.866025)$]{
		\begin{tikzpicture}[scale=4]
		\draw[dashed] (0,0) -- (0.4330,0.2500) -- (0.4330,0.7500) -- (0,1) -- (-0.4330,0.7500) -- (-0.4330,0.2500) -- cycle;
		\draw (0,0) -- (0,1);
		\draw (0.4330,0.2500) -- (-0.4330,0.7500);
		\draw (0.4330,0.7500) -- (-0.4330,0.2500);
		\end{tikzpicture}
	}
	\subfloat[$(\geo{R}_{3,6},0.965926)$]{
		\begin{tikzpicture}[scale=4]
		\draw[dashed] (0,0) -- (0.3660,0.3660) -- (0.5000,0.8660) -- (0,1) -- (-0.5000,0.8660) -- (-0.3660,0.3660) -- cycle;
		\draw (0,0) -- (0.5000,0.8660) -- (-0.5000,0.8660) -- cycle;
		\draw (0,0) -- (0,1);
		\draw (0.3660,0.3660) -- (-0.5000,0.8660);
		\draw (0.5000,0.8660) -- (-0.3660,0.3660);
		\end{tikzpicture}
	\label{figure:6gon:R36}
	}
	\caption{Two equilateral small $6$-gons $(\geo{P}_6,W(\geo{P}_6))$: (a) Regular $6$-gon; (b) Reinhardt $6$-gon~\cite{reinhardt1922}}
	\label{figure:6gon}
\end{figure}

\begin{figure}[H]
	\centering
	\subfloat[$(\geo{R}_8,0.923880)$]{
		\begin{tikzpicture}[scale=4]
		\draw[dashed] (0,0) -- (0.3536,0.1464) -- (0.5000,0.5000) -- (0.3536,0.8536) -- (0,1) -- (-0.3536,0.8536) -- (-0.5000,0.5000) -- (-0.3536,0.1464) -- cycle;
		\draw (0,0) -- (0,1);
		\draw (0.3536,0.1464) -- (-0.3536,0.8536);
		\draw (0.5000,0.5000) -- (-0.5000,0.5000);
		\draw (0.3536,0.8536) -- (-0.3536,0.1464);
		\end{tikzpicture}
	}
\subfloat[$(\geo{H}_8,0.950394)$]{
	\begin{tikzpicture}[scale=4]
		\draw[dashed] (0,1) -- (0.3796,0.9251) -- (0.5000,0.5574) -- (0.3228,0.2134) -- (0,0) -- (-0.3228,0.2134) -- (-0.5000,0.5574) -- (-0.3796,0.9251) -- cycle;
		\draw (0,0) -- (0,1);
		\draw (0,0) -- (0.3796,0.9251);\draw (0,0) -- (-0.3796,0.9251);
		\draw (0.3796,0.9251) -- (-0.3228,0.2134);\draw (-0.3796,0.9251) -- (0.3228,0.2134);
		\draw (0.5000,0.5574) -- (-0.5000,0.5574);
	\end{tikzpicture}
		\label{figure:8gon:H8}
}
\subfloat[$(\geo{F}_8,0.953776)$]{
	\begin{tikzpicture}[scale=4]
		\draw[dashed] (0,0) -- (0.3208,0.2140) -- (0.5000,0.5555) -- (0.3841,0.9233) -- (0,0.9576) -- (-0.3841,0.9233) -- (-0.5000,0.5555) -- (-0.3208,0.2140) -- cycle;
		\draw (0,0) -- (0.3841,0.9233) -- (-0.3208,0.2140);\draw (0,0) -- (-0.3841,0.9233) -- (0.3208,0.2140);
		\draw (0.5000,0.5555) -- (-0.5000,0.5555);
	\end{tikzpicture}
	\label{figure:8gon:F8}
}
\subfloat[$(\geo{B}_8,0.977609)$]{
	\begin{tikzpicture}[scale=4]
		\draw[dashed] (0,0) -- (0.2957,0.2043) -- (0.5000,0.5000) -- (0.4114,0.9114) -- (0,1) -- (-0.4114,0.9114) -- (-0.5000,0.5000) -- (-0.2957,0.2043) -- cycle;
		\draw (0,0) -- (0.4114,0.9114) -- (-0.5000,0.5000) -- (0.5000,0.5000) -- (-0.4114,0.9114) -- cycle;
		\draw (0,0) -- (0,1);
		\draw (0.4114,0.9114) -- (-0.2957,0.2043);\draw (-0.4114,0.9114) -- (0.2957,0.2043);
	\end{tikzpicture}
	\label{figure:8gon:B8}
}
\caption{Four small $8$-gons $(\geo{P}_8,W(\geo{P}_8))$: (a) Regular $8$-gon; (b) Equilateral $8$-gon of maximal perimeter~\cite{audet2004}; (c) Equilateral $8$-gon of maximal width; (d) A non-equilateral $8$-gon of maximal width~\cite{audet2013}}
\label{figure:8gon}
\end{figure}
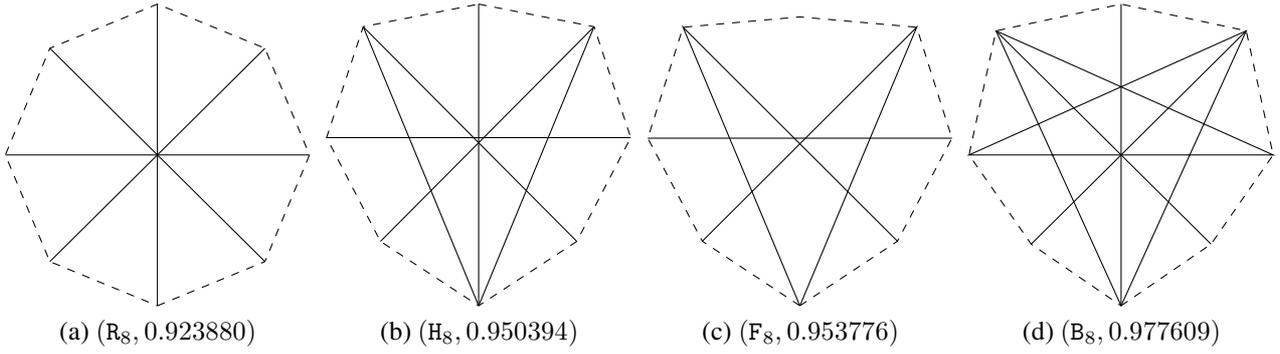


\section{Widths of small polygons}\label{sec:ngon}
Let $W(\geo{P})$ denote the width of a polygon $\geo{P}$. For a given integer $n\ge 3$, the regular small $n$-gon~$\geo{R}_n$ satisfies
\[
W(\geo{R}_n) =
\begin{cases}
	\cos \frac{\pi}{2n} &\text{if $n$ is odd,}\\
	\cos \frac{\pi}{n} &\text{if $n$ is even.}\\
\end{cases}
\]

When $n$ has an odd factor $m$, consider the family of equilateral small $n$-gons constructed as follows:
\begin{enumerate}
	\item Transform the regular small $m$-gon  $\geo{R}_m$ into a Reuleaux $m$-gon by replacing each edge by a circle's arc passing through its end vertices and centered at the opposite vertex;
	\item Add at regular intervals $n/m-1$ vertices within each arc;
	\item Take the convex hull of all vertices.
\end{enumerate}
These $n$-gons are denoted $\geo{R}_{m,n}$ and their widths satisfy $W(\geo{R}_{m,n}) = \cos \frac{\pi}{2n}$. The $6$-gon $\geo{R}_{3,6}$ is illustrated in Figure~\ref{figure:6gon:R36}.

\begin{theorem}[Bezdek and Fodor~\cite{bezdek2000}]\label{thm:width}
	For all $n \ge 3$, let $W_n^*$ denote the maximal width among all small $n$-gons, $w_n^*$ the maximal width among all equilateral ones, and $\ub{W}_n := \cos \frac{\pi}{2n}$.
	\begin{itemize}
		\item When $n$ has an odd factor, $w_n^* = W_n^* = \ub{W}_n$ is achieved by finitely many equilateral $n$-gons~\cite{mossinghoff2011,hare2013,hare2019}, including~$\geo{R}_{m,n}$. The optimal $n$-gon $\geo{R}_{m,n}$ is unique if $m$ is prime and $n/m \le 2$.
		\item When $n=2^s$ with integer $s\ge 2$, $W(\geo{R}_n) <  W_n^* < \ub{W}_n$.
	\end{itemize}
\end{theorem}

When $n = 2^s$, $W_n^*$ is known for $s \le 3$ and for the equilateral cases, $w_n^*$ is known only for $s = 2$. Bezdek and Fodor~\cite{bezdek2000} proved that $W_4^* = \sqrt{3}/2$, and this value is achieved by infinitely many non-equilateral small $4$-gons, including~$\geo{Q}_4$ represented in Figure~\ref{figure:4gon:Q4}. Audet, Hansen, Messine, and Ninin~\cite{audet2013} found that $W_8^* = \frac{1}{4}\sqrt{10+2\sqrt{7}}$, which is also achieved by infinitely many non-equilateral small $8$-gons, including~$\geo{B}_8$ represented in Figure~\ref{figure:8gon:B8}. For the equilateral $4$-gons, it is trivial that $w_4^* = W(\geo{R}_4) = \sqrt{2}/2 < W_4^*$.

For $n = 2^s$ with $s\ge 4$, tight lower bounds on $W_n^*$ may be obtained analytically. Bingane~\cite{bingane2021b,bingane2021d} proved that, for $n=2^s$ with $s\ge 4$,
\[
W_n^* \ge \cos \left(\frac{\pi}{2n} + \frac{1}{2}\arctan \left(\tan \frac{2\pi}{n}\tan \frac{\pi}{n}\right)-\frac{1}{2}\arcsin\left(\frac{\sin (2\pi/n) \sin (\pi/n)}{\sqrt{4\sin^2(\pi/n) + \cos (4\pi/n)}}\right)\right),
\]
which implies
\[
\ub{W}_n - W_n^* \le \frac{\pi^5}{4n^5} + O\left(\frac{1}{n^7}\right).
\]

\section{The equilateral small octagon of maximal width}\label{sec:F8}

\subsection{Definitions}
Cartesian coordinates are used to describe an $n$-gon $\geo{P}_n$. The vertex $\geo{v}_k$, $k\in \{0,1,\ldots,n-1\}$, is positioned at abscissa $x_k$ and ordinate $y_k$. Sums or differences of coordinate indices are taken modulo~$n$. Without any loss of generality, the vertex $\geo{v}_0$ is placed at the origin: $x_0 = y_0 = 0$, the $n$-gon $\geo{P}_n$ belongs to the half-plane $y\ge 0$ and the vertices $\geo{v}_k$, $k \in \{1,2,\ldots,n-1\}$, are arranged in a counterclockwise order as illustrated in Figure~\ref{figure:model}, i.e., $x_iy_{i+1} \ge y_ix_{i+1}$ for all $i \in \{1,2,\ldots,n-2\}$. The $n$-gon $\geo{P}_n$ is small if $\max_{i,j} \|\geo{v}_i - \geo{v}_j\| = 1$. It is equilateral if $\|\geo{v}_i - \geo{v}_{i-1}\| = c$ for all $i\in \{1,2,\ldots,n\}$.

\begin{figure}[h]
	\centering
	\begin{tikzpicture}[scale=4]
		\draw[dashed] (0,0) node[below]{$\geo{v}_0(0,0)$} -- (0.3228,0.2134) node[right]{$\geo{v}_1(x_1,y_1)$} -- (0.5000,0.5574) node[right]{$\geo{v}_2(x_2,y_2)$} -- (0.3796,0.9251) node[right]{$\geo{v}_3(x_3,y_3)$} -- (0,1) node[above]{$\geo{v}_4(x_4,y_4)$} -- (-0.3796,0.9251) node[left]{$\geo{v}_5(x_5,y_5)$} -- (-0.5000,0.5574) node[left]{$\geo{v}_6(x_6,y_6)$} -- (-0.3228,0.2134) node[left]{$\geo{v}_7(x_7,y_7)$} -- cycle;
		\draw[->] (-0.25,0)--(0.25,0)node[below]{$x$};
		\draw[->] (0,0)--(0,0.5)node[left]{$y$};
	\end{tikzpicture}
	\caption{Definition of variables: Case of $n=8$ vertices}
	\label{figure:model}
\end{figure}
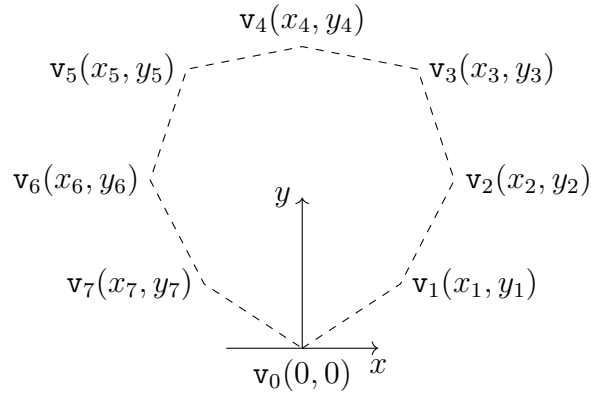

Let $i \in \{1,2,\ldots,n\}$. The distance $h_{ik}$ between a vertex $\geo{v}_k$, $k \in \{0,1,\ldots,n-1\}$, and the line containing the side $\geo{v}_{i-1}\geo{v}_i$ is
\[
h_{ik} := \frac{(x_{i-1}-x_k)(y_i-y_k) - (y_{i-1}-y_k)(x_i-x_k)}{\sqrt{(x_i - x_{i-1})^2 + (y_i - y_{i-1})^2}}.
\]
The {\em height $h_i$ associated to the side $\geo{v}_{i-1}\geo{v}_i$} is defined as the maximal value among $h_{ik}$ and we denote~$k_i$ the index of a vertex $\geo{v}_{k_i}$ such that $h_{ik_i} = h_i$. Then the width of the $n$-gon $\geo{P}_n$ is given by
\[
W(\geo{P}_n) = \min_{i =1,2,\ldots,n} h_i.
\]

We define the {\em height graph} of the $n$-gon $\geo{P}_n$ as the graph with the vertices $\geo{v}_0, \geo{v}_0, \ldots, \geo{v}_{n-1}$ of $\geo{P}_n$, and the edges $\geo{v}_{i-1}-\geo{v}_{k_i}$ and $\geo{v}_{i}-\geo{v}_{k_i}$ for all $i\in \{1,2,\ldots,n\}$. By construction, the length of an edge of the height graph is greater than or equal to $W(\geo{P}_n)$. Remark that the diameter graph of $\geo{P}_n$ is a subgraph of its height graph. Figure~\ref{figure:8gon:height} shows height graphs of some small polygons, the edges are represented by solid and dotted segments. The solid lines illustrate pairs of vertices which are unit distance apart. The dotted lines connect pairs of vertices that are a distance greater or equal to the width of the polygon but less than one. The dashed lines are not part of the graph, as they simply represent the sides of the polygons.

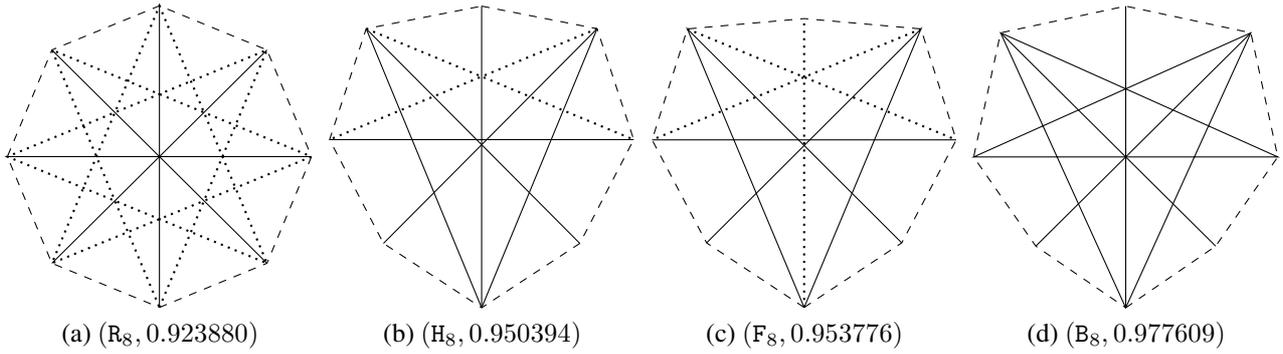
\begin{figure}[h]
	\centering
	\subfloat[$(\geo{R}_8,0.923880)$]{
		\begin{tikzpicture}[scale=4]
			\draw[dashed] (0,0) -- (0.3536,0.1464) -- (0.5000,0.5000) -- (0.3536,0.8536) -- (0,1) -- (-0.3536,0.8536) -- (-0.5000,0.5000) -- (-0.3536,0.1464) -- cycle;
			\draw (0,0) -- (0,1);
			\draw (0.3536,0.1464) -- (-0.3536,0.8536);
			\draw (0.5000,0.5000) -- (-0.5000,0.5000);
			\draw (0.3536,0.8536) -- (-0.3536,0.1464);
			\draw[dotted,thick] (0,0) -- (0.3536,0.8536) -- (-0.5000,0.5000) -- (0.3536,0.1464) -- (0,1) -- (-0.3536,0.1464) -- (0.5000,0.5000) -- (-0.3536,0.8536) -- cycle;
		\end{tikzpicture}
	}
	\subfloat[$(\geo{H}_8,0.950394)$]{
		\begin{tikzpicture}[scale=4]
			\draw[dashed] (0,1) -- (0.3796,0.9251) -- (0.5000,0.5574) -- (0.3228,0.2134) -- (0,0) -- (-0.3228,0.2134) -- (-0.5000,0.5574) -- (-0.3796,0.9251) -- cycle;
			\draw (0,0) -- (0,1);
			\draw (0,0) -- (0.3796,0.9251);\draw (0,0) -- (-0.3796,0.9251);
			\draw (0.3796,0.9251) -- (-0.3228,0.2134);\draw (-0.3796,0.9251) -- (0.3228,0.2134);
			\draw (0.5000,0.5574) -- (-0.5000,0.5574);
			\draw[dotted,thick] (0.3796,0.9251) -- (-0.5000,0.5574); \draw[dotted,thick] (-0.3796,0.9251) -- (0.5000,0.5574);
		\end{tikzpicture}
	}
	\subfloat[$(\geo{F}_8,0.953776)$]{
		\begin{tikzpicture}[scale=4]
			\draw[dashed] (0,0) -- (0.3208,0.2140) -- (0.5000,0.5555) -- (0.3841,0.9233) -- (0,0.9576) -- (-0.3841,0.9233) -- (-0.5000,0.5555) -- (-0.3208,0.2140) -- cycle;
			\draw (0,0) -- (0.3841,0.9233) -- (-0.3208,0.2140);\draw (0,0) -- (-0.3841,0.9233) -- (0.3208,0.2140);
			\draw (0.5000,0.5555) -- (-0.5000,0.5555);
			\draw[dotted,thick] (0,0) -- (0,0.9576);
			\draw[dotted,thick] (0.3841,0.9233) -- (-0.5000,0.5555); \draw[dotted,thick] (-0.3841,0.9233) -- (0.5000,0.5555);
		\end{tikzpicture}
	}
	\subfloat[$(\geo{B}_8,0.977609)$]{
		\begin{tikzpicture}[scale=4]
			\draw[dashed] (0,0) -- (0.2957,0.2043) -- (0.5000,0.5000) -- (0.4114,0.9114) -- (0,1) -- (-0.4114,0.9114) -- (-0.5000,0.5000) -- (-0.2957,0.2043) -- cycle;
			\draw (0,0) -- (0.4114,0.9114) -- (-0.5000,0.5000) -- (0.5000,0.5000) -- (-0.4114,0.9114) -- cycle;
			\draw (0,0) -- (0,1);
			\draw (0.4114,0.9114) -- (-0.2957,0.2043);\draw (-0.4114,0.9114) -- (0.2957,0.2043);
		\end{tikzpicture}
	}
	\caption{Height graphs of some small $8$-gons $(\geo{P}_8,W(\geo{P}_8))$: (a) Regular $8$-gon; (b) Equilateral $8$-gon of maximal perimeter~\cite{audet2004}; (c) Equilateral $8$-gon of maximal width; (d) A non-equilateral $8$-gon of maximal width~\cite{audet2013}}
	\label{figure:8gon:height}
\end{figure}

\subsection{Bounds on the maximal width}
A trivial lower bound on the maximal width $w_8^*$ is provided by the regular $8$-gon $\geo{R}_8$. However, the convex equilateral small $8$-gon of maximal perimeter $\geo{H}_8$~\cite{audet2004}, illustrated in Figure~\ref{figure:8gon:H8}, provides a tighter bound: $W(\geo{H}_8) = 0.9503943246\ldots > W(\geo{R}_8) = \frac{1}{2}\sqrt{2+\sqrt{2}}$. This lower bound is improved in Proposition~\ref{thm:F8:bound}.

\begin{proposition}\label{thm:F8:bound}
There exists an equilateral small $8$-gon $\geo{F}_8$ such that
	\[
	W(\geo{F}_8) = 0.9537763006\ldots > W(\geo{H}_8),
	\]
	where $\geo{H}_8$ is the convex equilateral small octagon from~\cite{audet2004} that maximizes the perimeter.
\end{proposition}
\begin{proof}
Consider an equilateral small $8$-gon $\geo{P}_8$ having the following height graph: a cycle $\geo{v}_{0} - \geo{v}_3 - \geo{v}_6 - \geo{v}_2 - \geo{v}_5 -\geo{v}_0$ of length five, plus three pendant edges $\geo{v}_{0} - \geo{v}_4$, $\geo{v}_3 - \geo{v}_7$, and $\geo{v}_5 - \geo{v}_1$, as illustred in Figure~\ref{figure:model:optimal}. In addition, $\geo{P}_8$ is axially symmetrical with respect to the edge $\geo{v}_{0}-\geo{v}_4$ and the three edges represented by dotted lines possess the same length $d$, i.e., $\|\geo{v}_{0}-\geo{v}_4\| = \|\geo{v}_{2}-\geo{v}_5\| = \|\geo{v}_{3}-\geo{v}_6\| = d \le 1$.

\begin{figure}
	\centering
	\begin{tikzpicture}[scale=8]
		\draw[dashed] (0,0) node[below]{$\geo{v}_0(0,0)$} -- (0.3208,0.2140) node[right]{$\geo{v}_1(x_1,y_1)$} -- (0.5000,0.5555) node[right]{$\geo{v}_2(x_2,y_2)$} -- (0.3841,0.9233) node[right]{$\geo{v}_3(x_3,y_3)$} -- (0,0.9576) node[above]{$\geo{v}_4(x_4,y_4)$} -- (-0.3841,0.9233) node[left]{$\geo{v}_5(x_5,y_5)$} -- (-0.5000,0.5555) node[left]{$\geo{v}_6(x_6,y_6)$} -- (-0.3208,0.2140) node[left]{$\geo{v}_7(x_7,y_7)$} -- cycle;
		\draw (0,0) -- (0.3841,0.9233) -- (-0.3208,0.2140);\draw (0,0) -- (-0.3841,0.9233) -- (0.3208,0.2140);
		\draw (0.5000,0.5555) -- (-0.5000,0.5555);
		\draw[dotted,thick] (0,0) -- (0,0.9576);
		\draw[dotted,thick] (0.3841,0.9233) -- (-0.5000,0.5555); \draw[dotted,thick] (-0.3841,0.9233) -- (0.5000,0.5555);
		\draw (0.0960,0.2308) arc (67.41:90.00:0.25) node[midway,above]{$\alpha$};
		\draw (0.2079,0.7460) arc (225.18:247.41:0.25) node[midway,below]{$\beta$};
		\draw (0.1631,0.8313) arc (202.59:225.18:0.24) node[midway,left]{$\alpha$};
		\draw (-0.3500,0.5555) arc (0.00:22.59:0.15) node[midway,right]{$\alpha$};
	\end{tikzpicture}
	\caption{Configuration of the $8$-gon providing a lower bound on the maximal width $w_8^*$}
	\label{figure:model:optimal}
\end{figure}
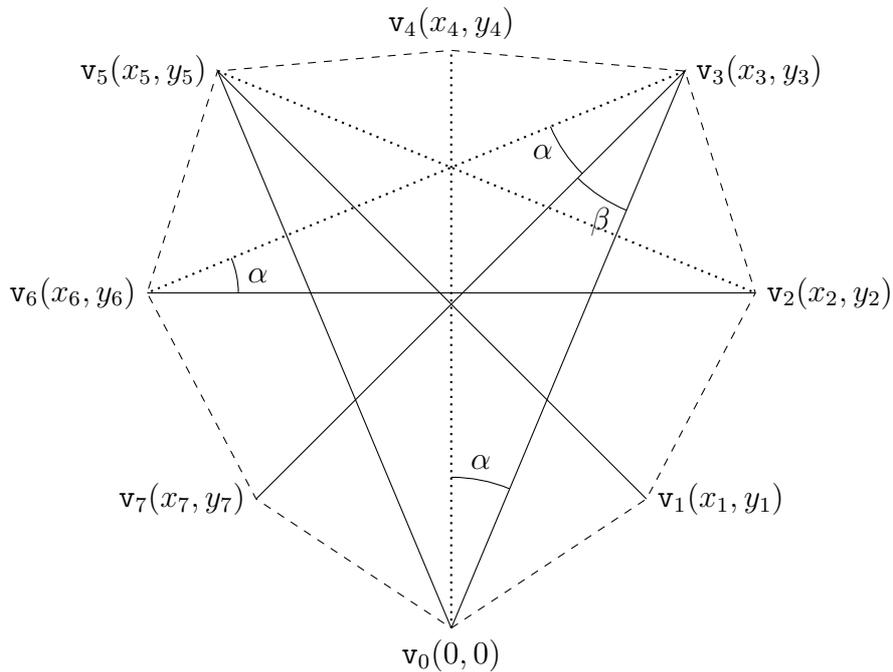

Place the vertex $\geo{v}_4$ at $(0,d)$ in the plane. Let $\alpha := \angle \geo{v}_4 \geo{v}_0 \geo{v}_3 \in (0,\pi/6)$ and $\beta := \angle \geo{v}_0 \geo{v}_3 \geo{v}_7 \in (0,\pi/8)$. Since $\geo{P}_8$ is equilateral, we have $\|\geo{v}_4-\geo{v}_3\| = \|\geo{v}_0-\geo{v}_7\|$, which implies
\begin{equation}\label{eq:F8:condition:dab}
	1 + d^2 - 2d\cos \alpha = 2 - 2 \cos \beta.
\end{equation}
Since $\geo{P}_8$ is symmetric, we have from the cycle $\geo{v}_0 - \geo{v}_3 - \geo{v}_6 - \geo{v}_2 - \geo{v}_5 - \geo{v}_{0}$,
\begin{equation}\label{eq:F8:condition:ab}
	3\alpha + \beta = \pi/2.
\end{equation}
Since the edge $\geo{v}_6 - \geo{v}_2$ is horizontal and $\|\geo{v}_6 - \geo{v}_2\| = 1$, we also have $x_6 = -1/2 = - x_2$, which yields
\begin{equation}\label{eq:F8:x6}
	d = \frac{2\sin \alpha +1}{2\sin(2\alpha + \beta)}.
\end{equation}

The system of equations~\eqref{eq:F8:condition:dab},~\eqref{eq:F8:condition:ab}, and~\eqref{eq:F8:x6} has a solution
\[
(\alpha_0,\beta_0,d_0) = (0.3942432313\ldots, 0.3880666326\ldots, 0.9575669263\ldots).
\]
We can show that $t_0 = 2\sin \alpha_0 = 0.7682191676\ldots$ is the unique positive root of the polynomial equation
\[
t^5-6t^3+3t^2+10t-7=0.
\]
Let $\geo{F}_8$ denote the $8$-gon obtained by setting $(\alpha,\beta,d) = (\alpha_0,\beta_0,d_0)$. The width of $\geo{F}_8$ is
\[
W(\geo{F}_8) = \frac{d_0\sin \alpha_0}{2\sin (\beta_0/2)} = 0.9537763006\ldots > W(\geo{H}_8).
\]
By construction, $\geo{F}_8$ is small and its vertices are given by
\[
\begin{aligned}
	x_0 &= 0,&&&y_0 &= 0,\\
	x_1 &= 0.3208100713\ldots &= -x_7, &&y_1 &= 0.2140003477\ldots &= y_7,\\
	x_2 &= 1/2 &= -x_6, &&y_2 &= 0.5554768772\ldots &= y_6,\\
	x_3 &= 0.3841095838\ldots &= -x_5, &&y_3 &= 0.9232875108\ldots &= y_5,\\
	x_4 &= 0, &&&y_4 &= 0.9575669263\ldots.
\end{aligned}
\]
\end{proof}

From Theorem~\ref{thm:width}, $w_8^* < W_8^* < \ub{W}_8$. A tighter upper bound is obtained by using Theorem~\ref{thm:width:perimeter}.

\begin{theorem}[Audet, Hansen, and Messine~\cite{audet2009b}]\label{thm:width:perimeter}
	For all $n \ge 3$, let $L(\geo{P}_n)$ denote the perimeter of an $n$-gon $\geo{P}_n$ and $W(\geo{P}_n)$ its width. Then $\frac{W(\geo{P}_n)}{L(\geo{P}_n)} \le \frac{1}{2n} \cot \frac{\pi}{2n}$.
\end{theorem}

Since the perimeter of any convex equilateral small octagon, including those of maximal width~$w_8^*$, is bounded above by $L(\geo{H}_8)$, it follows that
\[
w_8^* < \frac{L(\geo{H}_8)}{16} \cot \frac{\pi}{16} < 0.9727 < W_8^*,
\]
with $L(\geo{H}_8) = 3.0956093174\ldots > L(\geo{R}_8) = 4\sqrt{2-\sqrt{2}}$.

\subsection{Height graph}
Let $\geo{P}_8^*$ be an equilateral small $8$-gon of maximal width, i.e., $W(\geo{P}_8^*) = w_8^*$, and let $\lb{w} := W(\geo{F}_8)$. If $c$ denotes the sides length of $\geo{P}_8^*$ then
\[
c < \ub{c} := \frac{L(\geo{H}_8)}{8} < 0.3870.
\]
A lower bound on $c$ is deduced from Theorem~\ref{thm:width:perimeter}:
\[
c \ge 2w_8^* \tan \frac{\pi}{16} \ge 2\lb{w} \tan \frac{\pi}{16} > 0.3794.
\]

Next, we study the height graph of the optimal octagon~$\geo{P}_8^*$. Combining the upper bound on~$c$ with the lower bound on~$w_8^*$, one gets that the furthest vertex $\geo{v}_{k_i}$ to the side $\geo{v}_{i-1}\geo{v}_i$ must be the third neighbor of the side. This leads to the following lemma:
\begin{lemma}\label{lemma:ki}
	For all $i \in \{1,2,\ldots,8\}$, $k_i \in \{i+3,i+4\}$.
\end{lemma}
\begin{proof}
	Since $c < \ub{c} < 0.3870$ and $w_8^* \ge \lb{w} > 0.9537 > 2\ub{c}$, we have $k_i \not\in\{i+1, i+2, i+5,i+6\}$.
\end{proof}
The next lemma shows that two edges of the optimal height graph cannot be disjoint.
\begin{lemma}\label{lemma:ki:4gon}
	For all $i \in \{1,2,\ldots,8\}$, $\min \{\|\geo{v}_i-\geo{v}_{i+3}\|, \|\geo{v}_{i+4}-\geo{v}_{i-1}\|\} < w_8^*$.
\end{lemma}
\begin{proof}
	Without loss of generality, assume that for $i=1$, $\min\{\|\geo{v}_1-\geo{v}_4\|,\|\geo{v}_5-\geo{v}_0\|\} \ge w_8^* > 0.9537$. It follows that
	\[
	1 \ge \max\{\|\geo{v}_1-\geo{v}_4\|,\|\geo{v}_5-\geo{v}_0\|\} \ge \sqrt{{w_8^*}^2 + c^2} > \sqrt{0.9537^2 + 0.3794^2} > 1.0263.
	\]
	This contradiction implies that at most one of the two diagonals may exceed $w_8^*$.
\end{proof}

Combining Lemma~\ref{lemma:ki} and Lemma~\ref{lemma:ki:4gon}, one obtains additional conditions that the furthest vertices must satisfy. This leads to the following proposition:

\begin{proposition}\label{thm:ki:height}
	For all $i \in \{1,2,\ldots,8\}$, $k_i = i + 3 \Leftrightarrow k_{i+4} = i$ and $k_i = i + 4 \Leftrightarrow k_{i+4} = i-1$.
\end{proposition}
\begin{proof}
	Without loss of generality, assume $i=1$. From Lemma~\ref{lemma:ki}, $k_1 \in \{4,5\}$.
	\begin{enumerate}
		\item If $k_1 = 4$ then $\|\geo{v}_1-\geo{v}_4\| \ge w_8^*$. It follows, from Lemma~\ref{lemma:ki:4gon}, $\|\geo{v}_5-\geo{v}_0\| < w_8^*$, which implies $k_5 \not= 0$. Finally, from Lemma~\ref{lemma:ki}, $k_5 = 1$.
		\item If $k_5 = 1$ then $\|\geo{v}_1-\geo{v}_4\| \ge w_8^*$. It follows, from Lemma~\ref{lemma:ki:4gon}, $\|\geo{v}_5-\geo{v}_0\| < w_8^*$, which implies $k_1 \not= 5$. Finally, from Lemma~\ref{lemma:ki}, $k_1 = 4$.
	\end{enumerate}
A similar reasoning yields $k_1 = 5 \Leftrightarrow k_5 = 0$.
\end{proof}

The next corollary details the structure of the optimal height graph.
\begin{corollary}\label{thm:height:optimal}
	The height graph of $\geo{P}_8^*$ has either a cycle of length~$7$ with one pendant edge, or a cycle of length~$5$ with three pendant edges from three non-consecutive vertices of the cycle.
\end{corollary}
\begin{proof}
	The result follows from Proposition~\ref{thm:ki:height}.
\end{proof}

From Corollary~\ref{thm:height:optimal}, there are only two possible height graphs of $\geo{P}_8^*$, both represented in Figure~\ref{figure:8gon:height:optimal} and given by the vector of furthest vertices $\rv{k} = (k_1, k_2,\ldots,k_8)$: $k_i$ is the index of the furthest vertex to the side $\geo{v}_{i-1}\geo{v}_i$. The first configuration $\rv{k}^a =(5,6,7,0,0,1,2,3)$ in Figure~\ref{figure:8gon:height:optimal:k1} has a cycle of length~$7$ with one pendant edge and the second $\rv{k}^b =(5,5,6,0,0,2,3,3)$ in Figure~\ref{figure:8gon:height:optimal:k2} has a cycle of length~$5$ with three pendant edges.

\begin{figure}[h]
	\centering
	\subfloat[$\rv{k}^a=(5,6,7,0,0,1,2,3)$]{
		\begin{tikzpicture}[scale=4]
			\draw[dashed] (0,0) node[below]{$\geo{v}_0(0,0)$} -- (0.3536,0.1464) node[right]{$\geo{v}_1(x_1,y_1)$} -- (0.5000,0.5000) node[right]{$\geo{v}_2(x_2,y_2)$} -- (0.3536,0.8536) node[right]{$\geo{v}_3(x_3,y_3)$} -- (0,1) node[above]{$\geo{v}_4(x_4,y_4)$} -- (-0.3536,0.8536) node[left]{$\geo{v}_5(x_5,y_5)$} -- (-0.5000,0.5000) node[left]{$\geo{v}_6(x_6,y_6)$} -- (-0.3536,0.1464) node[left]{$\geo{v}_7(x_7,y_7)$} -- cycle;
			\draw[dotted,thick] (0,0) -- (0,1);
			\draw[dotted,thick] (0,0) -- (0.3536,0.8536) -- (-0.3536,0.1464) -- (0.5000,0.5000) -- (-0.5000,0.5000) -- (0.3536,0.1464) -- (-0.3536,0.8536) -- cycle;
		\end{tikzpicture}
	\label{figure:8gon:height:optimal:k1}
	}
	\subfloat[$\rv{k}^b=(5,5,6,0,0,2,3,3)$]{
		\begin{tikzpicture}[scale=4]
			\draw[dashed] (0,0) node[below]{$\geo{v}_0(0,0)$} -- (0.3208,0.2140) node[right]{$\geo{v}_1(x_1,y_1)$} -- (0.5000,0.5555) node[right]{$\geo{v}_2(x_2,y_2)$} -- (0.3841,0.9233) node[right]{$\geo{v}_3(x_3,y_3)$} -- (0,0.9576) node[above]{$\geo{v}_4(x_4,y_4)$} -- (-0.3841,0.9233) node[left]{$\geo{v}_5(x_5,y_5)$} -- (-0.5000,0.5555) node[left]{$\geo{v}_6(x_6,y_6)$} -- (-0.3208,0.2140) node[left]{$\geo{v}_7(x_7,y_7)$} -- cycle;
			\draw[dotted,thick] (0,0) -- (0.3841,0.9233) -- (-0.3208,0.2140);\draw[dotted,thick] (0,0) -- (-0.3841,0.9233) -- (0.3208,0.2140);
			\draw[dotted,thick] (0.5000,0.5555) -- (-0.5000,0.5555);
			\draw[dotted,thick] (0,0) -- (0,0.9576);
			\draw[dotted,thick] (0.3841,0.9233) -- (-0.5000,0.5555); \draw[dotted,thick] (-0.3841,0.9233) -- (0.5000,0.5555);
		\end{tikzpicture}
	\label{figure:8gon:height:optimal:k2}
	}
	\caption{Possible optimal height graphs $\rv{k} = (k_1,k_2, \ldots,k_8)$: (a) Configuration~$\rv{k}^a$; (b) Configuration~$\rv{k}^b$}
	\label{figure:8gon:height:optimal}
\end{figure}
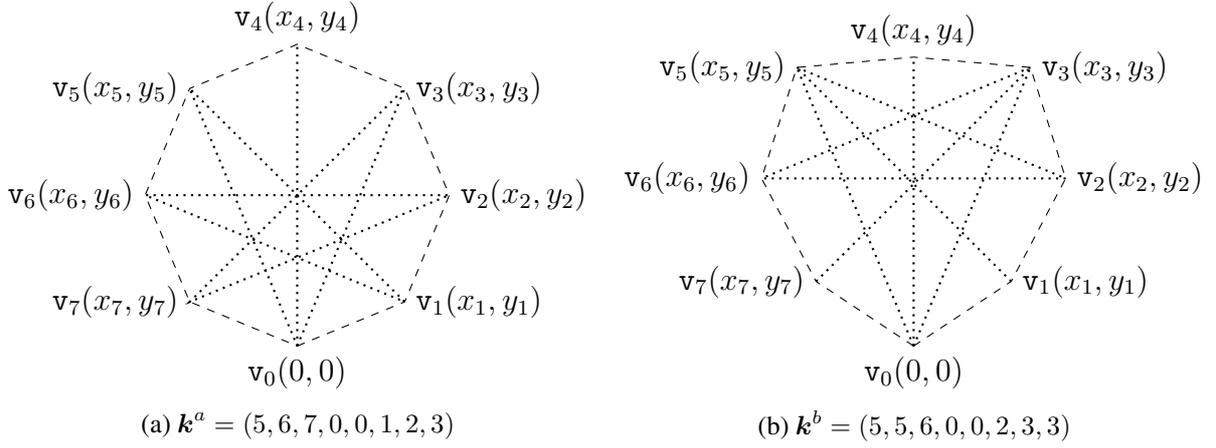

In both configurations $\rv{k}^a$ and $\rv{k}^b$, there exists an index $i \in \{1,2,\ldots,8\}$ such that $k_i = i + 3$ and $k_{i+3} = i-1$. This property allows to further tighten bounds of some attributes of the optimal octagon in Proposition~\ref{thm:ki:tight}.

\begin{proposition}\label{thm:ki:tight}
	For all $i \in \{1,2,\ldots,8\}$, if $k_i = i + 3$ and $k_{i+3} = i-1$ then $c > 0.3844$, $w_8^* < 0.9564$, and $\max\{\|\geo{v}_{i-1}-\geo{v}_{i+2}\|, \|\geo{v}_{i}-\geo{v}_{i+3}\|\} < 0.9644$.
\end{proposition}
\begin{proof}
	Without any loss of generality, assume that, for $i=1$, $k_1 = 4$ and $k_4 = 0$. Let $a := \|\geo{v}_0-\geo{v}_3\|$, $b := \|\geo{v}_1-\geo{v}_4\|$, and $d := \|\geo{v}_0-\geo{v}_4\|$. From the triangles $\geo{v}_0 \geo{v}_3 \geo{v}_4$ and $\geo{v}_4 \geo{v}_0 \geo{v}_1$, we have
	\[
	\begin{aligned}
		a^2 &= d^2 + c^2 - 2c\sqrt{d^2-h_4^2},\\
		b^2 &= d^2 + c^2 - 2c\sqrt{d^2-h_1^2},
	\end{aligned}
	\]
	respectively. Therefore
	\[
	\min \{a,b\} \ge \sqrt{d^2 + c^2 - 2c\sqrt{d^2-{w_8^*}^2}}.
	\]
	Applying Ptolemy's inequality to the quadrilateral $\geo{v}_0\geo{v}_1\geo{v}_3\geo{v}_4$ yields
	\[
	ab \le c^2 + \|\geo{v}_1-\geo{v}_3\|d.
	\]
	Since $\|\geo{v}_1-\geo{v}_3\| \le 2c$, it follows that
	\begin{itemize}
		\item $d^2 + c^2 - 2c\sqrt{d^2-{w_8^*}^2} \le c^2 + 2cd \Rightarrow c \ge \frac{d^2}{2d + 2\sqrt{d^2-{w_8^*}^2}} \ge \frac{1}{2 + 2\sqrt{1-\lb{w}^2}} > 0.3844$,
		\item $d^2 + c^2 - 2c\sqrt{d^2-{w_8^*}^2} \le c^2 + 2cd \Rightarrow w_8^* \le \frac{d\sqrt{4cd-d^2}}{2c} \le \frac{\sqrt{4\ub{c}-1}}{2\ub{c}} < 0.9564$,
		\item $\max \{a,b\} \le \frac{c^2 + 2cd}{\sqrt{d^2 + c^2 - 2c\sqrt{d^2-{w_8^*}^2}}} \le \frac{\ub{c}^2 + 2\ub{c}}{\sqrt{1 + \ub{c}^2 - 2\ub{c}\sqrt{1-\lb{w}^2}}} < 0.9644$.
	\end{itemize}
\end{proof}

\subsection{Numerical results}
Now, mathematical optimization tools are used to show that the octagon $\geo{F}_8$ is optimal. For $n = 8$, the equilateral case of the maximal width problem can be formulated as follows:
\begin{subequations}\label{eq:8gon:width}
	\begin{align}
		\max_{\rv{x},\rv{y},c,\rv{h},w} \quad & w\\
		\subj \quad & (x_j - x_i)^2 + (y_j - y_i)^2 \le 1 &\forall 1\le i < j \le 7,\label{eq:ngon:d}\\
		& x_i^2 + y_i^2 \le 1 &\forall 1 \le i \le 7,\label{eq:ngon:r}\\
		& y_i \ge 0 &\forall 1 \le i \le 7,\label{eq:ngon:y}\\
		& x_iy_{i+1} - y_ix_{i+1} \ge 0 &\forall 1 \le i \le 6,\label{eq:ngon:order}\\
		& c^2 =  (x_i - x_{i-1})^2 + (y_i - y_{i-1})^2 &\forall 1 \le i \le 8,\label{eq:ngon:vi}\\
		& w \le h_i &\forall 1 \le i \le 8,\label{eq:ngon:hi}\\
		& c h_i = (x_{i-1}-x_{k_i})(y_i-y_{k_i}) - (y_{i-1}-y_{k_i})(x_i-x_{k_i})	&\forall 1 \le i \le 8,\\
		& x_4 = 0,\\
		& 0.384462 \le c \le 0.386952,\\
		& w \ge 0.953776.
	\end{align}
\end{subequations}

Problem~\eqref{eq:8gon:width} was solved on the NEOS Server~6.0 using AMPL with Couenne~0.5.8. AMPL codes have been made available at \url{https://github.com/cbingane/optigon}. The solver Couenne~\cite{belotti2009} is a branch-and-bound algorithm that aims at finding global optima of nonconvex mixed-integer nonlinear optimization problems. The results support the following keypoints:
\begin{itemize}
	\item Under configuration $\rv{k}^a = (5,6,7,0,0,1,2,3)$, Couenne certified in 0.32 s that Problem~\eqref{eq:8gon:width} is infeasible, i.e., there is no equilateral small octagon under configuration $\rv{k}^a$ having a width greater than $\lb{w}$.
	\item Under configuration $\rv{k}^b = (5,5,6,0,0,2,3,3)$, Couenne solved Problem~\eqref{eq:8gon:width} in 3.18 s. The optimal value obtained is
	\[
	w^* = 0.95377630014\ldots
	\]
	and the vertices of the optimal octagon are
	\[
	\begin{aligned}
		x_0^* &= 0,&&&y_0^* &= 0,\\
		x_1^* &= 0.320810 &= -x_7^*, &&y_1^* &= 0.214000 &= y_7^*,\\
		x_2^* &= 0.500000 &= -x_6^*, &&y_2^* &= 0.555477 &= y_6^*,\\
		x_3^* &= 0.384110 &= -x_5^*, &&y_3^* &= 0.923288 &= y_5^*,\\
		x_4^* &= 0, &&&y_4^* &= 0.957567.
	\end{aligned}
	\]
	Within the limit of the numerical computations, this corresponds to the octagon~$\geo{F}_8$.
\end{itemize}

\begin{proposition}
	The width of an equilateral small octagon does not exceed the value $\lb{w} + 10^{-6}$.
\end{proposition}
\begin{proof}
	By replacing the lower bound $w\ge 0.953776$ on the variable $w$ by $w\ge 0.953777$ in Problem~\eqref{eq:8gon:width}, Couenne certified in 2.05 s that there is no equilateral small octagon whose width exceeds the value $0.953777$.
\end{proof}

Let $\rv{z}^*$ be an optimal solution of a maximization problem
\[
\max_{\rv{z} \in \Omega} g(\rv{z}),
\]
where $\Omega \subset \R^n$ and $g\colon \R^n \to \R$. For $\varepsilon^* \ge 0$, $\rv{z}^* \in \Omega$ is said {\em $\varepsilon^*$-unique} if
\[
g(\rv{z}^*) > \max_{\rv{z} \in \Omega, \|\rv{z} - \rv{z}^*\| \ge \varepsilon} g(\rv{z})
\]
for any $\varepsilon \ge \varepsilon^*$. This definition was introduced in~\cite{audet2004} to show that $\geo{H}_8$ is the near-unique convex equilateral small octagon of maximal perimeter.

\begin{proposition}
	The optimal octagon $\geo{F}_8$ is $\varepsilon^*$-unique, with $\varepsilon^* = 1.5 \times 10^{-4}$, using the Euclidian norm.
\end{proposition}
\begin{proof}
	Let $(x_i^*,y_i^*)$, $i = 0,1,\ldots,7$, be the vertices coordinates of $\geo{F}_8$. Solving with Couenne showed in 2.80 s that adding the constraint
	\[
	\sum_{i=1}^{7} (x_i - x_i^*)^2 + (y_i - y_i^*)^2 \ge 2.25 \times 10^{-8}
	\]
	to Problem~\eqref{eq:8gon:width} makes it infeasible due to the lower bound on the variable $w$.
\end{proof}

\section{Bounds on the maximal width of equilateral small $2^s$-gons with integer $s \ge 4$}\label{sec:Gn}
This section provides a family of equilateral small $n$-gons, with $n=2^s$ and $s\ge 4$, whose widths are within $O(1/n^4)$ of the optimal value.

For any $n=2^s$ with $s\ge 4$, consider the $n$-gon $\geo{G}_n$, illustrated in Figure~\ref{figure:Gn} for $n=16$ and $n=32$. Its height graph has the edge $\geo{v}_0-\geo{v}_{\frac{n}{2}}$ as axis of symmetry and can be described by  the $(3n/4-1)$-length cycle $\geo{v}_0 - \geo{v}_{\frac{n}{2}-1} - \ldots - \geo{v}_{\frac{3n}{4}+1} - \geo{v}_{\frac{n}{4}} - \geo{v}_{\frac{3n}{4}} - \geo{v}_{\frac{n}{4}-1} - \ldots - \geo{v}_{\frac{n}{2}+1} - \geo{v}_0$ and the pendant edges $\geo{v}_0 - \geo{v}_{\frac{n}{2}}$, $\geo{v}_{4k-1} - \geo{v}_{4k-1+\frac{n}{2}}$, $k=1,\ldots,n/8$. The edges in dotted lines have all the same length~$d$.

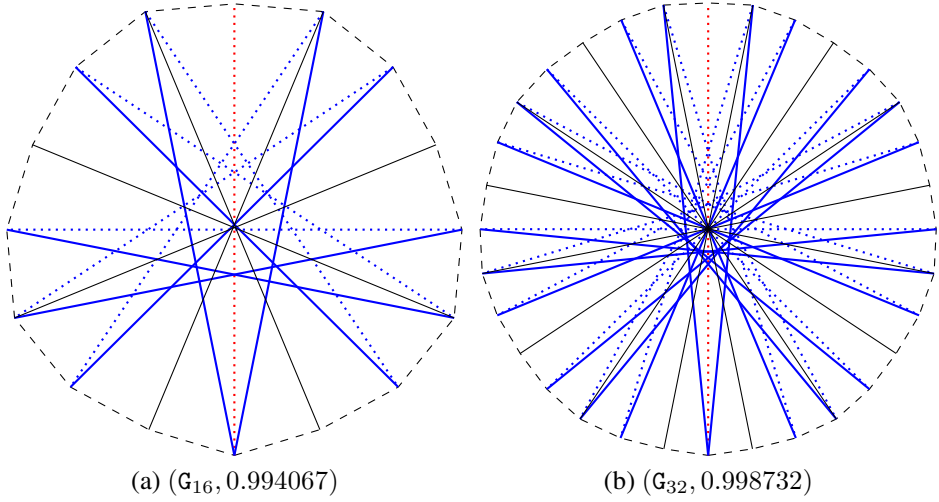
\begin{figure}[h]
	\centering
	\subfloat[$(\geo{G}_{16},0.994067)$]{
		\begin{tikzpicture}[scale=6]
			\draw[dashed] (0,0) -- (0.1875,0.0569) -- (0.3592,0.1510) -- (0.4818,0.3038) -- (0.4989,0.4989) -- (0.4421,0.6864) -- (0.3479,0.8582) -- (0.1951,0.9808) -- (0,0.9979) -- (-0.1951,0.9808) -- (-0.3479,0.8582) -- (-0.4421,0.6864) -- (-0.4989,0.4989) -- (-0.4818,0.3038) -- (-0.3592,0.1510) -- (-0.1875,0.0569) -- cycle;
			\draw[dotted,red,thick] (0,0)--(0,0.9979);
			\draw[blue,thick] (0,0) -- (0.1951,0.9808); \draw[blue,thick] (0,0) -- (-0.1951,0.9808);
			\draw[dotted,blue,thick] (0.1951,0.9808) -- (-0.3592,0.1510); \draw[dotted,blue,thick] (-0.1951,0.9808) -- (0.3592,0.1510);
			\draw[blue,thick] (-0.3592,0.1510) -- (0.3479,0.8582); \draw[blue,thick] (0.3592,0.1510) -- (-0.3479,0.8582);
			\draw[dotted,blue,thick] (0.3479,0.8582) -- (-0.4818,0.3038); \draw[dotted,blue,thick] (-0.3479,0.8582) -- (0.4818,0.3038);
			\draw[blue,thick] (-0.4818,0.3038) -- (0.4989,0.4989); \draw[blue,thick] (0.4818,0.3038) -- (-0.4989,0.4989);
			\draw[dotted,blue,thick] (0.4989,0.4989) -- (-0.4989,0.4989);
			\draw (0.1951,0.9808) -- (-0.1875,0.0569);\draw (-0.1951,0.9808) -- (0.1875,0.0569);
			\draw (-0.4818,0.3038) -- (0.4421,0.6864);\draw (0.4818,0.3038) -- (-0.4421,0.6864);
		\end{tikzpicture}
	}
	\subfloat[$(\geo{G}_{32},0.998732)$]{
		\begin{tikzpicture}[scale=6]
			\draw[dashed] (0,0) -- (0.0971,0.0144) -- (0.1922,0.0384) -- (0.2809,0.0804) -- (0.3535,0.1464) -- (0.4120,0.2253) -- (0.4623,0.3095) -- (0.4952,0.4019) -- (0.4999,0.4999) -- (0.4855,0.5970) -- (0.4616,0.6922) -- (0.4195,0.7808) -- (0.3535,0.8534) -- (0.2747,0.9119) -- (0.1905,0.9622) -- (0.0980,0.9952) -- (0,0.9999) -- (-0.0980,0.9952) -- (-0.1905,0.9622) -- (-0.2747,0.9119) -- (-0.3535,0.8534) -- (-0.4195,0.7808) -- (-0.4616,0.6922) -- (-0.4855,0.5970) -- (-0.4999,0.4999) -- (-0.4952,0.4019) -- (-0.4623,0.3095) -- (-0.4120,0.2253) -- (-0.3535,0.1464) -- (-0.2809,0.0804) -- (-0.1922,0.0384) -- (-0.0971,0.0144) -- cycle;
			\draw[dotted,red,thick] (0,0)--(0,0.9999);
			\draw[blue,thick] (0,0) -- (0.0980,0.9952); \draw[blue,thick] (0,0) -- (-0.0980,0.9952);
			\draw[dotted,blue,thick] (0.0980,0.9952) -- (-0.1922,0.0384); \draw[dotted,blue,thick] (-0.0980,0.9952) -- (0.1922,0.0384);
			\draw[blue,thick] (-0.1922,0.0384) -- (0.1905,0.9622); \draw[blue,thick] (0.1922,0.0384) -- (-0.1905,0.9622);
			\draw[dotted,blue,thick] (0.1905,0.9622) -- (-0.2809,0.0804); \draw[dotted,blue,thick] (-0.1905,0.9622) -- (0.2809,0.0804);
			\draw[blue,thick] (-0.2809,0.0804) -- (0.3535,0.8534); \draw[blue,thick] (0.2809,0.0804) -- (-0.3535,0.8534);
			\draw[dotted,blue,thick] (0.3535,0.8534) -- (-0.3535,0.1464); \draw[dotted,blue,thick] (-0.3535,0.8534) -- (0.3535,0.1464);
			\draw[blue,thick] (-0.3535,0.1464) -- (0.4195,0.7808); \draw[blue,thick] (0.3535,0.1464) -- (-0.4195,0.7808);
			\draw[dotted,blue,thick] (0.4195,0.7808) -- (-0.4623,0.3095); \draw[dotted,blue,thick] (-0.4195,0.7808) -- (0.4623,0.3095);
			\draw[blue,thick] (-0.4623,0.3095) -- (0.4616,0.6922); \draw[blue,thick] (0.4623,0.3095) -- (-0.4616,0.6922);
			\draw[dotted,blue,thick] (0.4616,0.6922) -- (-0.4952,0.4019); \draw[dotted,blue,thick] (-0.4616,0.6922) -- (0.4952,0.4019);
			\draw[blue,thick] (-0.4952,0.4019) -- (0.4999,0.4999); \draw[blue,thick] (0.4952,0.4019) -- (-0.4999,0.4999);
			\draw[dotted,blue,thick] (0.4999,0.4999) -- (-0.4999,0.4999);
			\draw (0.0980,0.9952) -- (-0.0971,0.0144);\draw (-0.0980,0.9952) -- (0.0971,0.0144);
			\draw (-0.2809,0.0804) -- (0.2747,0.9119);\draw (0.2809,0.0804) -- (-0.2747,0.9119);
			\draw (0.4195,0.7808) -- (-0.4120,0.2253);\draw (-0.4195,0.7808) -- (0.4120,0.2253);
			\draw (-0.4952,0.4019) -- (0.4855,0.5970);\draw (0.4952,0.4019) -- (-0.4855,0.5970);
		\end{tikzpicture}
	}
	\caption{Polygons $(\geo{G}_n,W(\geo{G}_n))$ defined in Theorem~\ref{thm:Gn}: (a) Hexadecagon $\geo{G}_{16}$; (b) Triacontadigon $\geo{G}_{32}$}
	\label{figure:Gn}
\end{figure}

Place the vertex $\geo{v}_{\frac{n}{2}}$ at $(0,d)$ in the plane. Let $\alpha := \angle \geo{v}_{\frac{n}{2}} \geo{v}_0 \geo{v}_{\frac{n}{2}-1}$ and $\beta := \angle \geo{v}_0 \geo{v}_{\frac{n}{2}-1} \geo{v}_{n-1}$. Since $\geo{G}_n$ is equilateral, we have $\|\geo{v}_{\frac{n}{2}}-\geo{v}_{\frac{n}{2}-1}\| = \|\geo{v}_0-\geo{v}_{n-1}\|$, which implies
\begin{equation}\label{eq:condition:dab}
	1 + d^2 - 2d\cos \alpha = 2 - 2 \cos \beta.
\end{equation}
Since $\geo{G}_n$ is symmetric, we have from the cycle $\geo{v}_0 - \ldots - \geo{v}_{\frac{n}{4}} - \geo{v}_{\frac{3n}{4}} - \ldots - \geo{v}_{0}$,
\begin{equation}\label{eq:condition:ab}
\frac{n}{8}(3\alpha + \beta) = \frac{\pi}{2}.
\end{equation}
Since the edge $\geo{v}_{\frac{n}{4}} - \geo{v}_{\frac{3n}{4}}$ is horizontal and $\|\geo{v}_{\frac{n}{4}} - \geo{v}_{\frac{3n}{4}}\| = d$, we also have $x_{\frac{n}{4}} = \frac{d}{2} = - x_{\frac{3n}{4}}$, which yields after simplications,
\begin{equation}\label{eq:xn4}
	d = \frac{2\cos(2\alpha + \beta)+1}{2\cos \alpha + \cos(3\alpha + \beta)}.
\end{equation}
An asymptotic analysis produces that, for large $n$, the system of equations~\eqref{eq:condition:dab},~\eqref{eq:condition:ab}, and~\eqref{eq:xn4} has a solution $(\alpha_0(n),\beta_0(n),d_0(n))$ satisfying
\[
\begin{aligned}
	\alpha_0(n) &=  \frac{\pi}{n} + \frac{\pi^5}{6n^5} + \frac{\pi^7}{12n^7} + O\left(\frac{1}{n^9}\right),\\
	\beta_0(n) &=  \frac{\pi}{n} - \frac{\pi^5}{2n^5} - \frac{\pi^7}{4n^7} + O\left(\frac{1}{n^9}\right),\\
	d_0(n) &=  1 - \frac{4\pi^4}{3n^4} - \frac{7\pi^6}{3n^6} + O\left(\frac{1}{n^8}\right).
\end{aligned}
\]

By setting $(\alpha,\beta,d) = (\alpha_0(n),\beta_0(n),d_0(n))$, the width of $\geo{G}_n$ is
\[
	W(\geo{G}_n) = \frac{d_0(n)\sin \alpha_0(n)}{2\sin (\beta_0(n)/2)} = 1 - \frac{\pi^2}{8n^2} - \frac{85\pi^4}{128n^4} - \frac{92801\pi^6}{46080n^6} + O\left(\frac{1}{n^8}\right)
\]
and
\[
\begin{aligned}
	\ub{W}_n - W(\geo{G}_n) &= \frac{2\pi^4}{3n^4} + \frac{145\pi^6}{72n^6} + O\left(\frac{1}{n^8}\right),\\
	W(\geo{G}_n) - W(\geo{R}_n) &= \frac{3\pi^2}{8n^2} - \frac{271\pi^4}{384n^4} + O\left(\frac{1}{n^6}\right).
\end{aligned}
\]
By construction, $\geo{G}_n$ is small. This completes the proof of Theorem~\ref{thm:Gn}.\qed

All polygons presented in this work and in~\cite{bingane2021b,bingane2021c,bingane2021d,bingane2021f} were implemented as a MATLAB package: OPTIGON, which is freely available on GitHub~\cite{optigon}. One can also find an algorithm developed in~\cite{bingane2021a} to find an estimate of the maximal area of a small $n$-gon when $n \ge 6$ is even.

Table~\ref{table:width} gives the widths of $\geo{G}_n$, along with the upper bounds $\ub{W}_n$ and the widths of $\geo{R}_n$. We also report the widths of $\geo{F}_n$, illustrated in Figure~\ref{figure:Fn}. Polygons $\geo{F}_n$ can be seen as a generalization of the optimal $8$-gon $\geo{F}_8$ and we can show that
\[
\ub{W}_n - W(\geo{F}_n) = \frac{\pi^3}{2n^3} - \frac{5\pi^5}{8n^5} + O\left(\frac{1}{n^6}\right)
\]
for all $n=2^s$ and $s\ge 3$. The height graph of~$\geo{F}_n$ in Figure~\ref{figure:Fn} has the vertical edge as axis of symmetry and can be described by a cycle of length $n/2+1$, plus $n/2-1$ additional pendant edges, arranged so that all but two particular vertices of the cycle have a pendant edge. In Table~\ref{table:width}, $\geo{G}_n$ provides a tighter lower bound on the maximal width~$w_n^*$ compared to~$\geo{R}_n$ or $\geo{F}_n$. The fraction $\frac{W(\geo{G}_n) - W(\geo{R}_n)}{\ub{W}_n - W(\geo{R}_n)}$ of the length of the interval $[W(\geo{R}_n), \ub{W}_n]$ containing $W(\geo{G}_n)$ shows that $W(\geo{G}_n)$ approaches $\ub{W}_n$ much faster than $W(\geo{R}_n)$ does as $n$ increases. Indeed, $W(\geo{G}_n) - W(\geo{R}_n) \sim 3\pi^2/(8n^2)$ for large~$n$.

\begin{table}[t]
	\footnotesize
	\centering
	\caption{Widths of $\geo{G}_n$}
	\label{table:width}
	\begin{tabular}{@{}rllllr@{}}
		\toprule
		$n$ & $W(\geo{R}_n)$ & $W(\geo{F}_n)$ & $W(\geo{G}_n)$ & $\ub{W}_n$ & $ \frac{W(\geo{G}_n) - W(\geo{R}_n)}{\ub{W}_n - W(\geo{R}_n)}$ \\
		\midrule
		16	&	0.9807852804	&	0.9915310059	&	0.9940673080	&	0.9951847267	&	0.9224	\\
		32	&	0.9951847267	&	0.9983271244	&	0.9987316811	&	0.9987954562	&	0.9823	\\
		64	&	0.9987954562	&	0.9996398418	&	0.9996949197	&	0.9996988187	&	0.9957	\\
		128	&	0.9996988187	&	0.9999173147	&	0.9999244595	&	0.9999247018	&	0.9989	\\
		256	&	0.9999247018	&	0.9999802514	&	0.9999811602	&	0.9999811753	&	0.9997	\\
		\bottomrule
	\end{tabular}
\end{table}

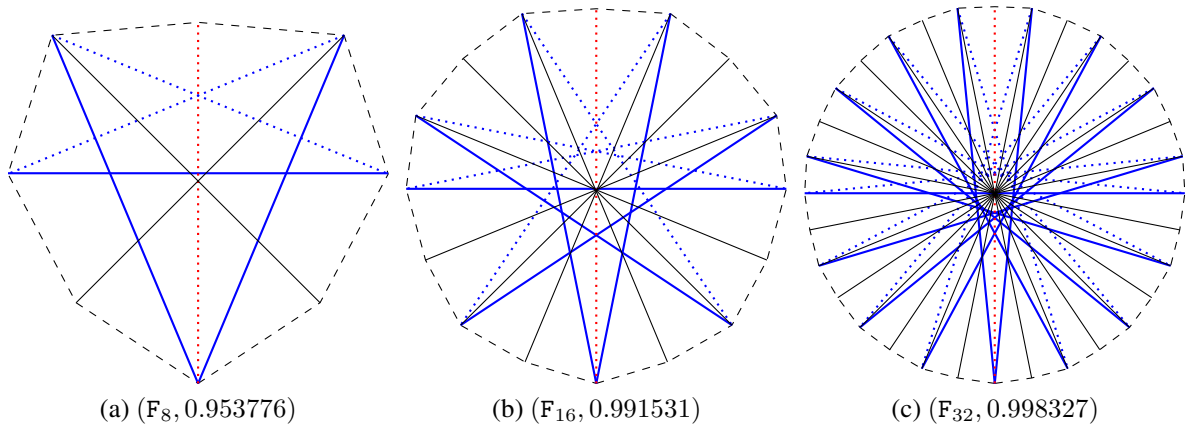
\begin{figure}[h]
	\centering
	\subfloat[$(\geo{F}_8,0.953776)$]{
		\begin{tikzpicture}[scale=5]
			\draw[dashed] (0,0) -- (0.3208,0.2140) -- (0.5000,0.5555) -- (0.3841,0.9233) -- (0,0.9576) -- (-0.3841,0.9233) -- (-0.5000,0.5555) -- (-0.3208,0.2140) -- cycle;
			\draw[blue,thick] (0,0) -- (0.3841,0.9233); \draw[blue,thick] (0,0) -- (-0.3841,0.9233);
			\draw[dotted,thick,blue] (0.3841,0.9233) -- (-0.5000,0.5555); \draw[dotted,thick,blue] (-0.3841,0.9233) -- (0.5000,0.5555);
			\draw[blue,thick] (0.5000,0.5574) -- (-0.5000,0.5574);
			\draw[dotted,red,thick] (0,0) -- (0,0.9576);
			\draw (0.3841,0.9233) -- (-0.3208,0.2140);\draw (-0.3841,0.9233) -- (0.3208,0.2140);
		\end{tikzpicture}
	}
	\subfloat[$(\geo{F}_{16},0.991531)$]{
		\begin{tikzpicture}[scale=5]
			\draw[dashed] (0,0) -- (0.1873,0.0568) -- (0.3569,0.1545) -- (0.4491,0.3271) -- (0.5000,0.5161) -- (0.4746,0.7102) -- (0.3501,0.8618) -- (0.1953,0.9807) -- (0,0.9937) -- (-0.1953,0.9807) -- (-0.3501,0.8618) -- (-0.4746,0.7102) -- (-0.5000,0.5161) -- (-0.4491,0.3271) -- (-0.3569,0.1545) -- (-0.1873,0.0568) -- cycle;
			\draw[blue,thick] (0,0) -- (0.1953,0.9807); \draw[blue,thick] (0,0) -- (-0.1953,0.9807);
			\draw[dotted,blue,thick] (0.1953,0.9807) -- (-0.3569,0.1545); \draw[dotted,blue,thick] (-0.1953,0.9807) -- (0.3569,0.1545);
			\draw[blue,thick] (-0.3569,0.1545) -- (0.4746,0.7102); \draw[blue,thick] (0.3569,0.1545) -- (-0.4746,0.7102);
			\draw[dotted,blue,thick] (0.4746,0.7102) -- (-0.5000,0.5161); \draw[dotted,blue,thick] (-0.4746,0.7102) -- (0.5000,0.5161);
			\draw[blue,thick] (0.5000,0.5161) -- (-0.5000,0.5161);
			\draw[dotted,red,thick] (0,0) -- (0,0.9937);
			\draw (0.1953,0.9807) -- (-0.1873,0.0568);\draw (-0.1953,0.9807) -- (0.1873,0.0568);
			\draw (-0.3569,0.1545) -- (0.3501,0.8618);\draw (0.3569,0.1545) -- (-0.3501,0.8618);
			\draw (0.4746,0.7102) -- (-0.4491,0.3271);\draw (-0.4746,0.7102) -- (0.4491,0.3271);
		\end{tikzpicture}
	}
	\subfloat[$(\geo{F}_{32},0.998327)$]{
		\begin{tikzpicture}[scale=5]
			\draw[dashed] (0,0) -- (0.0971,0.0144) -- (0.1920,0.0391) -- (0.2762,0.0895) -- (0.3545,0.1486) -- (0.4129,0.2274) -- (0.4626,0.3120) -- (0.4865,0.4072) -- (0.5000,0.5044) -- (0.4943,0.6023) -- (0.4613,0.6948) -- (0.4185,0.7830) -- (0.3526,0.8558) -- (0.2794,0.9210) -- (0.1907,0.9630) -- (0.0980,0.9952) -- (0,0.9991) -- (-0.0980,0.9952) -- (-0.1907,0.9630) -- (-0.2794,0.9210) -- (-0.3526,0.8558) -- (-0.4185,0.7830) -- (-0.4613,0.6948) -- (-0.4943,0.6023) -- (-0.5000,0.5044) -- (-0.4865,0.4072) -- (-0.4626,0.3120) -- (-0.4129,0.2274) -- (-0.3545,0.1486) -- (-0.2762,0.0895) -- (-0.1920,0.0391) -- (-0.0971,0.0144) -- cycle;
			\draw[blue,thick] (0,0) -- (0.0980,0.9952); \draw[blue,thick] (0,0) -- (-0.0980,0.9952);
			\draw[dotted,blue,thick] (0.0980,0.9952) -- (-0.1920,0.0391); \draw[dotted,blue,thick] (-0.0980,0.9952) -- (0.1920,0.0391);
			\draw[blue,thick] (-0.1920,0.0391) -- (0.2794,0.9210); \draw[blue,thick] (0.1920,0.0391) -- (-0.2794,0.9210);
			\draw[dotted,blue,thick] (0.2794,0.9210) -- (-0.3545,0.1486); \draw[dotted,blue,thick] (-0.2794,0.9210) -- (0.3545,0.1486);
			\draw[blue,thick] (-0.3545,0.1486) -- (0.4185,0.7830); \draw[blue,thick] (0.3545,0.1486) -- (-0.4185,0.7830);
			\draw[dotted,blue,thick] (0.4185,0.7830) -- (-0.4626,0.3120); \draw[dotted,blue,thick] (-0.4185,0.7830) -- (0.4626,0.3120);
			\draw[blue,thick] (-0.4626,0.3120) -- (0.4943,0.6023); \draw[blue,thick] (0.4626,0.3120) -- (-0.4943,0.6023);
			\draw[dotted,blue,thick] (0.4943,0.6023) -- (-0.5000,0.5044); \draw[dotted,blue,thick] (-0.4943,0.6023) -- (0.5000,0.5044);
			\draw[blue,thick] (0.5000,0.5044) -- (-0.5000,0.5044);
			\draw[dotted,red,thick] (0,0) -- (0,0.9991);
			\draw (0.0980,0.9952) -- (-0.0971,0.0144);\draw (-0.0980,0.9952) -- (0.0971,0.0144);
			\draw (-0.1920,0.0391) -- (0.1907,0.9630);\draw (0.1920,0.0391) -- (-0.1907,0.9630);
			\draw (0.2794,0.9210) -- (-0.2762,0.0895);\draw (-0.2794,0.9210) -- (0.2762,0.0895);
			\draw (-0.3545,0.1486) -- (0.3526,0.8558);\draw (0.3545,0.1486) -- (-0.3526,0.8558);
			\draw (0.4185,0.7830) -- (-0.4129,0.2274);\draw (-0.4185,0.7830) -- (0.4129,0.2274);
			\draw (-0.4626,0.3120) -- (0.4613,0.6948);\draw (0.4626,0.3120) -- (-0.4613,0.6948);
			\draw (0.4943,0.6023) -- (-0.4865,0.4072);\draw (-0.4943,0.6023) -- (0.4865,0.4072);
		\end{tikzpicture}
	}
	\caption{A generalization of the optimal octagon $(\geo{F}_n,W(\geo{F}_n))$: (a) Optimal octagon~$\geo{F}_{8}$; (b) Hexadecagon~$\geo{F}_{16}$; (c) Triacontadigon~$\geo{F}_{32}$}
	\label{figure:Fn}
\end{figure}

\section{Conclusion}\label{sec:conclusion}
We found with precision $10^{-6}$ the equilateral small octagon of maximal width and showed that there is no other optimal octagon outside a ball of radius $1.5 \times 10^{-4}$. In addition, for each $n=2^s$ where $s\ge 4$ is an integer, we constructed an equilateral small $n$-gon whose width is within $2\pi^4/(3n^4) + O(1/n^6)$ of the maximal width and exceeds that of the regular polygon.

\bibliographystyle{ieeetr}
\bibliography{../../research}
\end{document}